\documentclass[a4paper,12pt]{article}
\usepackage{amssymb}
\usepackage{mathrsfs}
\usepackage{amsmath}
\usepackage{amsfonts,amsthm,amssymb}
\usepackage{graphics}
\usepackage{color}
\usepackage{graphicx}
\usepackage{epstopdf}
\usepackage{subfigure}
\usepackage{changepage}
\usepackage{booktabs}

\usepackage[english]{babel}
\usepackage{bm}
\usepackage{graphicx}
\usepackage[natural]{xcolor}
\usepackage{lineno}

\newtheorem{theorem}{Theorem}
\newtheorem{corollary}{Corollary}

\newtheorem{proposition}{Proposition}
\newtheorem{lemma}{Lemma}
\newtheorem{remark}{Remark}
\newtheorem{fact}{Fact}
\newtheorem{example}{Example}

\newtheorem{definition}{Definition}

\def \T{\textup{T}}

\def\span{\textup{span}\,}
\def \r{\textup{rank\,}}
\def \br{\textup{rank}_2\,}
\def \pr{\textup{rank}_p\,}
\def\diag{\textup{diag}}
\def\RO{\textup{RO}}
\def\S{\textup{S}}

\begin{document}
	\title{Generalized spectral characterizations of almost controllable graphs}
	\author{\small	Wei Wang$^{\rm a}$\quad Fenjin Liu$^{\rm b}$\thanks{Corresponding author: fenjinliu@yahoo.com} \quad Wei Wang$^{\rm c}$
		\\
		{\footnotesize$^{\rm a}$School of Mathematics, Physics and Finance, Anhui Polytechnic University, Wuhu 241000, P.R. China}\\
		{\footnotesize$^{\rm b}$School of Science,
			Chang'an University,
			Xi'an 710064, P.R. China}\\
		{\footnotesize$^{\rm b}$School of Mathematics and Statistics,
			Xi'an Jiaotong University,
			Xi'an  710049, P.R. China}
	}
	\date{}
	
 \maketitle
\begin{abstract}
Characterizing graphs by their spectra is an important topic in spectral graph theory, which has attracted a lot of attention of researchers in recent years.
It is generally very hard and challenging to show a given graph to be determined by its spectrum. In Wang~[J. Combin. Theory, Ser. B, 122 (2017):
438-451], the author gave a simple arithmetic condition for a family of graphs being determined by their generalized spectra. However, the method applies only to a family of the so called \emph{controllable graphs}; it fails when the graphs are non-controllable.

  In this paper, we introduce a class of non-controllable graphs, called \emph{almost controllable graphs}, and prove that, for any pair of almost controllable graphs $G$ and $H$ that are generalized cospectral, there exist exactly two rational orthogonal matrices $Q$ with constant row sums such that $Q^{\rm T}A(G)Q=A(H)$, where $A(G)$ and $A(H)$ are the adjacency matrices of $G$ and $H$, respectively. The main ingredient of the proof is a use of the Binet-Cauchy formula. As an application, we obtain a simple criterion for an almost controllable graph $G$ to be determined by its generalized spectrum, which in some sense extends the corresponding result for controllable graphs. \\

\noindent\textbf{Keywords:} Graph spectra; Cospectral graphs; Rational orthogonal matrix; Controllable graph. \\

\noindent\textbf{Mathematics Subject Classification:} 05C50
\end{abstract}
\section{Introduction}
\label{intro}
Let $G$ be a simple graph  with vertex set $V(G)=\{1,2,\ldots,n\}$. The \emph{adjacency matrix} of $G$ is the $n\times n$ real symmetric matrix $A(G)=(a_{ij})$ (simply for $A$) where $a_{ij}=1$ if $i$ and $j$ are adjacent;  $a_{ij}=0$ otherwise. The  \emph{spectrum} of a graph $G$, denoted by $\sigma(G)$, is
the multiset of all the $n$ eigenvalues  of the matrix $A(G)$. The
\emph{generalized spectrum} of a graph $G$ is defined to be the pair $(\sigma(G),\sigma(\overline{G}))$, where $\overline{G}$ is the complement of $G$. Two graphs $G$ and $H$ are called \emph{cospectral} (resp.\emph{ generalized cospectral}) if $G$ and $H$ share the same spectrum (resp. generalized spectrum).

It is well-known that for any two cospectral graphs $G$ and $H$, their adjacency matrices $A(G)$ and $A(H)$ are similar via an orthogonal matrix $Q$, that is,  $Q^\T A(G) Q=A(H)$.
An orthogonal matrix $Q$ is \emph{regular} if each row sum of $Q$ is 1, that is, $Qe=e$, where $e$ is the all-ones vector. Note that all regular orthogonal matrices of a fixed order $n$ constitute a group under the ordinary matrix multiplication. We call it \emph{regular orthogonal group} and denote it by $\RO_n$. A fundamental result, due to Johnson and Newman \cite{johnson1980JCTB}, states that the adjacency matrices of two generalized cospectral graphs are similar {\color{red}via} a regular orthogonal matrix.
\begin{theorem}\cite{johnson1980JCTB}\label{genconQ} Two graphs $G$ and $H$ on the same vertex set $\{1,2,\ldots,n\}$  are generalized cospectral if and only if there exists a $Q\in \RO_n$ such that $Q^\T A(G)Q = A(H)$.
\end{theorem}
In general, for two generalized cospectral graphs $G$ and $H$, the matrix $Q$ satisfying $Q^\T A(G)Q = A(H)$ may not be unique. Indeed, if both $G$ and $H$ are empty graphs (i.e., $E(G)=E(H)=\emptyset$), then $Q$ can take any matrix in $\RO_n$.

A key observation of Wang~\cite{wang2006EUJC} is that under a suitable restriction, the corresponding matrix $Q$ is unique and rational. For a graph $G$ with $n$ vertices,  the \emph{walk matrix} of $G$ is
$$W(G):=[e,A(G)e,\ldots,A^{n-1}(G)e].$$ A graph $G$ is \emph{controllable} if  $W(G)$ is invertible. We use  $\mathcal{G}_n$ to denote the set of all controllable graphs
with $n$ vertices.
\begin{theorem}\cite{wang2006EUJC}\label{unique} Let $G\in \mathcal{G}_n$ and $H$ be a graph generalized cospectral to $G$. Then $H\in \mathcal{G}_n$ and there exists a unique matrix $Q\in \RO_n$ such that $Q^\T A(G) Q=A(H)$. Moreover, the unique $Q$ satisfies $Q=W(G)W^{-1}(H)$ and hence is rational.
\end{theorem}
A graph $G$ is \emph{determined by the generalized spectrum} (DGS for short) if, any graph generalized cospectral with $G$ is isomorphic to $G$. For controllable graphs $G$, using Theorem \ref{unique} as the starting point, Wang \cite{wang2013EJC,wang2017JCTB} obtained some sufficient conditions for $G$ to be DGS. The overall idea can be described as follows.

Let $\RO_n(\mathbb{Q})$ be the group consisting of all rational regular orthogonal $n\times n$ matrices, and $\S_n$ be the group of   $n\times n$ permutation matrices. Note that for any $Q\in \S_n$, the graph with  adjacency matrix $Q^\T A(G)Q$ is isomorphic to $G$. So in order to show that a graph $G\in \mathcal{G}_n$ is DGS, it suffices to show that any $Q\in \RO_n(\mathbb{Q})$ such that each entry  $Q^\T A(G)Q$ belongs to $\{0,1\}$, is a permutation matrix. Such an argument uses essentially the rationality of $Q$ and hence relies heavily on the controllability assumption of $G$.

The key finding of this paper is that there exists a particular subset of noncontrollable graphs for which the corresponding $Q$ must be rational. We will describe this particular subset in Sec.~3;  such graphs are called to be almost controllable. The first main result is Theorem \ref{main}, which, roughly speaking, extends the rationality in Theorem \ref{unique} to almost controllable graphs. In Sec.~4, we introduce a specific family $\mathcal{F}_n$ from the set of symmetric and almost controllable graphs.  The second main result is Theorem \ref{mainapp}, which presents a sufficient condition for a graph $G\in \mathcal{F}_n$  to be DGS.  Using Theorem \ref{main} as the new starting point, the full proof of Theorem \ref{mainapp} is given in Sec.~5. Some examples are given to illustrate  Theorem \ref{mainapp} in the final section.
\section{Preliminaries}
In this section, we cite some known results that will be used later in the paper. We begin with the well-known Binet-Cauchy formula.
\begin{lemma}\cite[pp. 8--9]{Gantmacher}\label{Binet-Cauchy}
Suppose that a square matrix $C=(c_{ij})_{m\times m}$ is the product of two rectangular matrices $A=(a_{ik})_{m\times n}$ and $B=(b_{kj})_{n\times m}$, i.e.,
$c_{ij}=\sum_{s=1}^na_{is}b_{sj}$ $(i,j=1,2,\dots,m)$. Then the determinant of $C$ is the sum of the products of all possible minors of the maximal ($m$-th) order of $A$ into the corresponding minors of the same order of $B$, i.e.,
$$\begin{vmatrix}
  c_{11}&c_{12}&\cdots&c_{1m}\\
  c_{21}&c_{22}&\cdots&c_{2m}\\
  \vdots&\vdots&&\vdots\\
  c_{m1}&c_{m2}&\cdots&c_{mm}\\
\end{vmatrix}=\sum_{1\le k_1<k_2<\cdots<k_m\le n}\begin{vmatrix}
  a_{1k_1}&a_{1k_2}&\cdots&a_{1k_m}\\
  a_{2k_1}&a_{2k_2}&\cdots&a_{2k_m}\\
  \vdots&\vdots&&\vdots\\
  a_{mk_1}&a_{mk_2}&\cdots&a_{mk_m}\\
\end{vmatrix}\begin{vmatrix}
  b_{k_1 1}&b_{k_1 2}&\cdots&b_{k_1 m}\\
  b_{k_2 1}&b_{k_2 2}&\cdots&b_{k_2 m}\\
  \vdots&\vdots&&\vdots\\
  b_{k_m 1}&b_{k_m 2}&\cdots&b_{k_m m}\\
\end{vmatrix}.$$
\end{lemma}

For a graph $G$, an eigenvalue is
\emph{main} if it has an associated eigenvector not orthogonal to the all-ones vector $e$.  The \emph{main polynomial} is $m_G(x)=(x-\mu_1)(x-\mu_2)\cdots(x-\mu_m)$, where $\mu_1,\ldots,\mu_m$ are all main eigenvalues of $G$. It is known \cite{rowlinson2007AADM} that $m_G(x)\in \mathbb{Z}[x]$. The next four lemmas state some properties of the walk matrix of a graph.

For an integer matrix $M$ and a prime $p$, the $p$-rank of $M$, denote by $\pr M$, is the rank of $M$ over finite field $\mathbb{F}_p$. We use $\r M$ to denote the ordinary rank of $M$, that is, the rank of $M$ over $\mathbb{Q}$.
\begin{lemma}\cite{hagos2002} \label{rankmain}
	The rank of the walk matrix of a graph $G$ is equal to the number of its main eigenvalues.
\end{lemma}
\begin{lemma}\cite{hagos2002} \label{intcom}
	Let $r=\r W(G)$. Then $m_G(A)e=0$ and hence $A^{r}e$ can be written as a linear combination of $e,Ae,\ldots,A^{r-1}e$ with integral coefficients.
\end{lemma}
Lemma  \ref{intcom} clearly indicates that the first $r$ columns constitute a basis for the column space of $W(G)$ over $\mathbb{Q}$. Similar conclusion also holds over finite field $\mathbb{F}_p$; a proof can be found in \cite{qiu2019DM}.
\begin{lemma}\cite{qiu2019DM}\label{first}
	Let $G$ be a graph and $p$ be a prime. Let $r = \pr W(G)$. Then the first $r$ columns of $W(G)$ constitute a basis for the column space of  $W(G)$ over $\mathbb{F}_p$.
\end{lemma}
\begin{lemma}\cite{wang2006PHD,wang2013EJC}\label{r2up}
	For any graph with $n$ vertices,
	$\br W(G)\le \lceil\frac{n}{2}\rceil$.
\end{lemma}
\begin{lemma}\cite{wang2006PHD,wang2013EJC}\label{even}
	For any graph with $n$ vertices, $e^\T A^k(G)e$ is even for any positive integer $k$.
\end{lemma}
Recall that an $n\times n$ matrix $U$ with integer entries is called \emph{unimodular} if $\det U = \pm 1$. Two $n\times n$ integral matrices $M_1$ and $M_2$ are called \emph{unimodular equivalent} if there exist two unimodular matrices $U$ and $V$ such that $M_1=UM_2V$. It is well known that  every integral  matrix $M$ is unimodular equivalent to some diagonal matrix $S=\diag(d_1,d_2,\ldots,d_n)$,  where $d_1, d_2, \ldots, d_n$ are nonnegative integers with $d_i | d_{i+1}$ for $i = 1, 2, ..., n-1$. The matrix $S$ is called the \emph{Smith normal form} (SNF for short) of $M$. The diagonal entries $d_1,d_2,\ldots,d_n$ are the \emph{invariant factors} of $M$. They are unique and, in fact, $d_i={D_i}/{D_{i-1}}$ for $i=1,2,\ldots,r$, where $D_0=1$ and $D_i$ is the $i$-th \emph{determinant divisor} of $M$, i.e., the greatest common divisor of all $i\times i$ minors of $M$. We note that many basic properties of a matrix can be obtained from its invariant factors, as described below.
\begin{fact}\label{basicsnf}
	Let $M$ be an $n\times n$ integral nonzero matrix with invariant factors $d_1,d_2,\ldots,d_n$. Let $p$ be any prime. Then the followings hold.
	
	\noindent	\textup{(\romannumeral1)}	 $\det M=\pm d_1 d_2\cdots d_n.$
	
	\noindent	\textup{(\romannumeral2)}
	$\pr M=\max\{i\colon\, d_i\not\equiv 0\pmod{p}\}$ and $\r M=\max\{i\colon\,d_i\neq 0\}$.
	
	\noindent	\textup{(\romannumeral3)}
	$p^{n-\pr M}\mid \det M$.
	
	\noindent	\textup{(\romannumeral4)}
	$Mx\equiv 0\pmod{p^2}$ has a solution $x\not\equiv 0\pmod{p}$ if and only if $p^2\mid d_n$.
\end{fact}
We remark that the last assertion will play a key role in Sec. \ref{oddp}. We refer to \cite{wang2013EJC} for a proof of this assertion.
\section{Almost controllable graphs}
We call a graph $G$ with $n$ vertices \emph{almost controllable} if the walk matrix $W(G)$ has rank $n-1$. Note that for any graph $G$, the number of orbits of $V(G)$ under the automorphism group $Aut(G)$ is at least $\r W(G)$ since two vertices in the same orbit have the same row vector in $W(G)$. Thus, if $G$ is  almost controllable, then $G$ has either $n-1$ or $n$ orbits under its automorphism group.   We use $\mathcal{H}_n$ to denote the set of all almost controllable graphs of order $n$.  The set $\mathcal{H}_n$ has a natural bipartition according to numbers of orbits.  We use  $\mathcal{H}^a_n$ to collect graphs in $\mathcal{H}_n$ which has $n$ orbits, that is, asymmetric graphs in $\mathcal{H}_n$. The set of remaining graphs in $\mathcal{H}_n$ are denoted by $\mathcal{H}^s_n$. Table 1 records the sizes of $\mathcal{H}_n^a$ and $\mathcal{H}_n^s$ for small $n$. For a graph $G$ in $\mathcal{H}^s_n$, we call two vertices $\tau$ and $\tau'$ twin vertices, or twins, if $\{\tau,\tau'\}$ is an orbit of $G$. Note that each graph in $\mathcal{H}^s_n$ has exactly a pair of twins.

\begin{table}[htbp]
	\footnotesize
	\centering
	\caption{\label{computer} Almost controllable graphs of small order}
	\begin{tabular}{lccccccccccccccc}
		\toprule
		Order $n$ of graphs &2&&3 && 4&&5&&6&&7&&8&&9\\
		$\#$ of  graphs   & 2&&4&&11&& 34&& 156&&1044&& 12346&&274668\\
		$|\mathcal{H}_n|$  &2&&2&&2&&6&&22&&214&&3100&&86578\\
		$|\mathcal{H}_n^{a}|$ & 0&&0&&0&&0&&0&&42&&926&&36552\\
	$|\mathcal{H}_n^{s}|$ &2&&2&&2&&6&&22&&172&&2174&&50026\\
		\bottomrule
	\end{tabular}
\end{table}

The following definition based on walk matrix $W(G)$ is crucial for almost controllable graphs.
\begin{definition}\label{defxi}
	Let $G\in \mathcal{H}_n$. Write
 \begin{equation*}
\xi(G)=\left(\begin{matrix}
W_{1n}(G)\\
W_{2n}(G)\\
\vdots\\
W_{nn}(G)\\
\end{matrix}
\right) \text{~and~} W_\delta(G)=[e,A(G)e,\ldots,A^{n-2}(G)e,(-1)^\delta\frac{1}{2^{\lfloor\frac{n}{2}\rfloor-1}}\cdot \xi(G)], \delta\in\{0,1\},
\end{equation*}
where $W_{i,n}(G)$ denotes the algebraic cofactor of the $(i,n)$-entry for $W(G)$.
\end{definition}
Before we establish the general properties of the new defined matrices $W_0(G)$ and $W_1(G)$, we consider a small example.
\begin{example}\textup{
	Let $G$ be the graph with adjacency matrix
	 \begin{equation*}
	A=\left(\begin{matrix}
0&1&0&0&0\\
	1&0&1&0&0\\
	0&1&0&1&1\\
	0&0&1&0&0\\
	0&0&1&0&0\\
	\end{matrix}
	\right).
	\end{equation*}
Now,  \begin{equation*}
W(G)=[e,Ae,A^2e,A^3e,A^4e]=\left(\begin{matrix}
1&1&2&4&6\\
1&2&4&6&14\\
1&3&4&10&14\\
1&1&3&4&10\\
1&1&3&4&10\\
\end{matrix}
\right),
\end{equation*}
and  $\xi(G)=(0,0,0,2,-2)^\T$. Thus,
 \begin{equation*}
W_0(G)=\left(\begin{matrix}
1&1&2&4&0\\
1&2&4&6&0\\
1&3&4&10&0\\
1&1&3&4&1\\
1&1&3&4&-1\\
\end{matrix}
\right) \text{~and~}
W_1(G)=\left(\begin{matrix}
1&1&2&4&0\\
1&2&4&6&0\\
1&3&4&10&0\\
1&1&3&4&-1\\
1&1&3&4&1\\
\end{matrix}
\right).
\end{equation*}
}	\end{example}
\begin{proposition}\label{vectxi}
	Let $G\in \mathcal{H}_n$. Then $\xi(G)$ given in Definition \ref{defxi} has the following properties.
	
\noindent	\textup{(\romannumeral1)} $\frac{1}{2^{\lfloor\frac{n}{2}\rfloor-1}}\cdot \xi(G)$ is a nonzero integral vector.

\noindent	\textup{(\romannumeral2)} $W^\T (G)\xi(G)=0.$
	
\noindent	\textup{(\romannumeral3)} $\xi^\T (G) \xi(G)=\det V^\T V$ where $V=[e,A(G)e,\ldots,A^{n-2}(G)e]$.
	
\noindent	\textup{(\romannumeral4)} If $H$ is a graph generalized cospectral with $G$, then $H\in \mathcal{H}_n$ and  $\xi^\T (H) \xi(H)=\xi^\T (G) \xi(G)$.
	
\end{proposition}
\begin{proof}
Let  $V=[e,A(G)e,
\ldots,A^{n-2}(G)e]$ and we write $X_i$ ($i=1,2,\ldots,n$)  for the matrix obtained from $V$ by removing row $i$.  As $X_i$ is a square matrix of order $n-1$, we have
$$2^{(n-1)-\br X_i}\mid \det X_i.$$
 Noting that $X_i$ is a submatrix of $W(G)$, we have $\br X_i\le \br W(G)$ and hence $\br X_i\le \lceil\frac{n}{2}\rceil$ by Lemma \ref{r2up}. Now,
$(n-1)-\br X_i\ge \lfloor\frac{n}{2}\rfloor-1$ and hence
$2^{\lfloor\frac{n}{2}\rfloor-1}\mid \det X_i$. Recall that the $i$-entry of $\xi(G)$ is the algebraic cofactor $W_{i,n}$ of the $(i,n)$-entry for $W(G)$. Clearly, $W_{i,n}$ is $(-1)^{n+i}\det X_i$ and hence $\frac{1}{2^{\lfloor\frac{n}{2}\rfloor-1}}\cdot \xi(G)$  is an integral vector. Furthermore, as $G\in \mathcal{H}_n$, we have $\r W(G)=n-1$ by the very definition. It follows from Lemma \ref{intcom} that the first $n-1$ columns of $W(G)$ are linearly independent, that is, $\r V=n-1$. Thus, there exists some  $i\in\{1,2,\ldots,n\}$ such that  $X_i$ is nonsingular. This  indicates that  $\xi(G)$ is nonzero and hence (\romannumeral1)  is proved.

Let $W^*(G)$ be the adjoint matrix of $W(G)$. Note that $\xi^\T(G)$ is the last row vector of $W^*(G)$. By the familiar equality $W^*W =(\det W)\cdot I_n$, where $I_n$ is the identity matrix, we easily obtain that
$\xi^\T(G)W(G)=(0,0,\ldots,0,\det W(G))$. Taking transpose and noting that $\det W(G)=0$ as $G\in \mathcal{H}_n$, we have $W^\T(G)\xi(G)=0$ and (\romannumeral2)  is proved.	

Note that $V$ is an  $n\times (n-1)$ matrix and $V^\T$ is an  $(n-1)\times n$ matrix. By the well-known Binet-Cauchy formula in Lemma \ref{Binet-Cauchy}, we have
\begin{equation}
\det V^\T V=\sum_{i=1}^n \det X_i^\T \det X_i.
\end{equation}
On the other hand, as $\xi(G)=(W_{1,n},\ldots,W_{n,n})^\T$ and $W_{i,n}=(-1)^{n+i}\det X_i$, we have
\begin{equation}
\xi^\T(G)\xi(G)=\sum_{i=1}^n W^2_{i,n}=\sum_{i=1}^n \det X_i \det X_i.
\end{equation}
This proves (\romannumeral3) as $\det X_i^\T=\det X_i$ always holds.

As $H$ is generalized cospectral with $G$, it follows from Theorem \ref{genconQ} that there is a matrix $Q\in\RO_n$ such that $Q^\T A(G) Q=A(H)$.  As $Q^\T Q=I_n$ and $Qe=e$, we have
 $A^k(H)e=Q^\T A^k(G)Qe=Q^\T A^k(G)e$ for any $k\ge 0$. Thus $W(H)=Q^\T W(G)$ and hence $\r W(H)=\r W(G)=n-1$, i.e., $H\in \mathcal{H}_n$. Write   $S=[e,A(H)e,\ldots,A^{n-2}(H)e]$. We have $S=Q^\T[e,A(G)e,\ldots,A^{n-2}(G)e]=Q^\T V$ and hence $S^\T S=V^\T V$. Note that $H\in \mathcal{H}_n$. Using (\romannumeral3) for $G$ and $H$, we clearly have $\xi^\T(G)\xi(G)=\xi^\T(H)\xi(H)$. This proves (\romannumeral4).
\end{proof}
\begin{corollary}\label{r2}
	If $G\in \mathcal{H}_n$, then 	$\r W_\delta(G)=n$ for $\delta=0,1$.
\end{corollary}

\begin{proof}
	Expanding the determinant of $W_\delta(G)$ by the last column, we find that $\det W_\delta(G)=(-1)^\delta2^{1-\lfloor{\frac{n}{2}}\rfloor}\xi^\T(G)\xi(G)$, which is nonzero  by Proposition \ref{vectxi}(\romannumeral1). The corollary follows.  	
\end{proof}

We are ready to present the main result of this paper, which can be regarded as a natural extension of Theorem \ref{unique} for almost controllable graphs.
\begin{theorem}\label{main}
	 Let $G\in \mathcal{H}_n$ and $H$ be a graph generalized cospectral to $G$. Then the equation $Q^\T A(G)Q=A(H)$ for variable $Q\in \RO_n$ has exactly two solutions $Q_0=W_0(G)W_{0}^{-1}(H) $ and $Q_1=W_1(G)W_0^{-1}(H)$, both of which are rational.
\end{theorem}
\begin{proof}
	As $G$ and $H$ are generalized cospectral, Theorem \ref{genconQ} indicates that the equation $Q^\T A(G)Q=A(H)$ has at least one solution $Q\in \RO_n$. Let $Q=\hat{Q}$ be any solution of $Q^\T A(G)Q=A(H)$ in $\RO_n$. Note that $H\in\mathcal{H}_n$  by Proposition \ref{vectxi}(\romannumeral4). We claim that  $\hat{Q}W_0(H)=W_0(G)$ or $W_1(G)$.
	
	As  $\hat{Q}^\T A(G)\hat{Q}=A(H)$, we have  $\hat{Q}^\T W(G)=W(H)$, i.e., $W^\T(G) \hat{Q}=W^\T (H)$. As $H\in \mathcal{H}_n$, we have $W^\T (H)\xi(H)=0$ by Proposition \ref{vectxi}(\romannumeral2). Thus, $W^\T(G) \hat{Q}\xi(H)=0$. Since $\r W^\T (G)=n-1$, the nullspace $N(W^\T(G))$  of  $W^\T(G)$ is one dimensional. By  Proposition \ref{vectxi}(\romannumeral1-\romannumeral2), we know that
	$\xi(G)$ is a nonzero vector in $N(W^\T(G))$ and hence $\hat{Q}\xi(H)=x \xi(G)$ for some real number $x$. By Proposition \ref{vectxi}(\romannumeral4), the two vectors
	$\xi(H),\xi(G)$ have the same length. Thus we must have $x=\pm 1$. Now the equality  $\hat{Q}\xi(H)=\pm \xi(G)$, together with the fact that $\hat{Q}W(H)=W(G)$, implies  $\hat{Q}W_0(H)=W_0(G)$ or $W_1(G)$, as claimed.
	
	From the equality $\hat{Q}W_0(H)=W_0(G)$ or $W_1(G)$, we know that $\hat{Q}=W_0(G)W_0^{-1}(H)$ or $W_1(G)W_0^{-1}(H)$. From the arbitrariness of $\hat{Q}$,
	we find that the  equation $Q^\T A(G)Q=A(H),Q\in\RO_n$ has at most two solutions. It remains to show that this equation has at least two solutions.
	
	First consider the special equation $Q^\T A(G) Q=A(G)$ for the case $H=G$. By Lemma \ref{rankmain}, the graph $G$ has exactly $n-1$ main eigenvalues. Therefore, there exist $n$ orthonormal eigenvectors $p_1,p_2,\ldots,p_n$ such that $e^\T p_n=0$ and $e^\T p_i\neq 0$ for $i=1,2,\ldots, n-1$. Let $P_\delta=[p_1,p_2,\ldots,p_{n-1},(-1)^\delta p_n]$ for $\delta=0,1$. Clearly, both $P_0$ and $P_1$ are orthogonal, and $P_0^\T A(G) P_0= P_1^\T A(G) P_1$. Let $P=P_0P_1^\T$. Then $P$ is orthogonal and $P^\T A(G) P=A(G)$. As $p_n^\T e=0$, one easily checks that $P_1^\T e=P_0^\T e$ and hence $Pe=e$. Thus $P\in \RO_n$ and $P$ is clearly not the identity matrix $I_n$. This proves that the equation  $Q^\T A(G) Q=A(G)$ has at least two solutions.

 Now consider the general equation $Q^\T A(G) Q=A(H)$ for two generalized cospectral graphs $G$ and $H$. We already know that the equation $Q^\T A(G) Q=A(H)$ has at least one solution, say $Q=\bar{Q}$, by Theorem \ref{genconQ}. Let $Q=P$ be a solution of $Q^\T A(G) Q=A(G)$, where $P\neq I_n$. One easily sees that $Q=P\bar{Q}$ also satisfies the equation  $Q^\T A(G) Q=A(H)$. Note that $P\bar{Q}\in \RO_n$ and $P\bar{Q}\neq \bar{Q}$. This indicates that the equation  $Q^\T A(G) Q=A(H)$ has at least two solutions and hence completes the proof of the theorem.
\end{proof}

A pair of generalized cospectral almost controllable graphs with two regular orthogonal matrices was reported in Example 10 \cite{liu2017EJC}, Theorem~\ref{main} gives a full explanation of this phenomenon.

\begin{definition} Let $Q\in \RO_n(\mathbb{Q})$. The level of $Q$, denoted by $\ell(Q)$, or simply $\ell$, is the smallest
positive integer $k$ such that $kQ$ is an integral matrix.
\end{definition}

Define $\mathcal{Q}_G(H)=\{Q\in \RO_n(\mathbb{Q})\colon\,Q^\T A(G) Q=A(H) \}$ and  $\mathcal{Q}_G=\cup_{H\in \mathcal{H}_n}\mathcal{Q}_G(H)$. Let $G$ and $H$ be two generalized cospectral graphs in $\mathcal{H}_n$. We note that the two matrices in $\mathcal{Q}_G(H)$, affirmed by Theorem \ref{main}, may not have the same level. In particular, if $G=H$ and $G$ is asymmetric, then one matrix in  $\mathcal{Q}_G(H)$ is the identity matrix whose level is 1, but the other matrix in $\mathcal{Q}_G(H)$ cannot be a permutation matrix (as $G$ is asymmetric) and hence has level at least 2.
\begin{proposition}\label{samelevel}
	If $G$ is a graph in $\mathcal{H}_n^s$ and $H$ is generalized cospectral with $G$, then the two matrices in  $\mathcal{Q}_G(H)$ have the same level.	
\end{proposition}
\begin{proof}
	As $G\in \mathcal{H}_n^s$, there exists a permutation matrix $P$, other than the identity matrix $I_n$, such that $P^\T A(G) P=A(G)$. Let $Q$ be a matrix in  $\mathcal{Q}_G(H)$. Clearly, $PQ\in \mathcal{Q}_G(H)$ and hence $\mathcal{Q}_G(H)=\{Q,PQ\}$. As the rows in $PQ$ are  rearrangement of  rows in $P$, the two matrices must have the same level and hence the proposition follows.
\end{proof}
An immediate corollary of Proposition \ref{samelevel} is the following criterion for a graph $G\in \mathcal{H}_n^s$ to be DGS. The same criterion for controllable graphs was presented in \cite{wang2006EUJC}.
\begin{corollary}\label{dgsperm}
Let $G\in\mathcal{H}_n^s$. Then $G$ is DGS if and only if $\mathcal{Q}_G$ contains only permutation matrices.
\end{corollary}
In the next section, we shall introduce a specific subset $\mathcal{F}_n$ of $\mathcal{H}_n^s$. It turns out that, for graphs $G$ in $\mathcal{F}_n$, we can give a simple condition to ensure that $G$ is DGS.
\section{A simple criterion for $G\in\mathcal{H}_n^s$ to be DGS}
For controllable graphs, Wang \cite{wang2017JCTB} gave the following simple arithmetic criterion for a graph to be DGS.
\begin{theorem}\label{controDGS}\cite{wang2017JCTB}
	If $\frac{\det W(G)}{2^{\lfloor\frac{n}{2}\rfloor}}$ is odd and square-free, then $G$ is DGS.
\end{theorem}
The determinant condition of Theorem \ref{controDGS} can be equivalently described using the language of Smith normal forms.
\begin{proposition}\label{detsnf}\cite{wang2017JCTB}
	$\frac{\det W(G)}{2^{\lfloor\frac{n}{2}\rfloor}}$ is odd and square-free if and only if the SNF of $W(G)$ is
	$$\diag[\underbrace{1,1,\ldots,1}_{\lceil\frac{n}{2}\rceil},\underbrace{2,2,\ldots,2,2b}_{\lfloor\frac{n}{2}\rfloor}],$$
	where $b$ is an odd square-free integer.
\end{proposition}
Now we introduce two particular subsets of $\mathcal{H}_n^s$.
\begin{definition}
	Let $\mathcal{F}_n=
	\{G\in\mathcal{H}_n^s\colon\,$ the SNF of W(G) is $$\diag[\underbrace{1,1,\ldots,1}_{\lceil\frac{n}{2}\rceil},\underbrace{2,2,\ldots,2,2b,0}_{\lfloor\frac{n}{2}\rfloor}],$$where $b$ is odd and square-free$\}.$  We use $\mathcal{F}_n^{*}$ to denote the subset of $\mathcal{F}_n$ such that the corresponding (n-1)-th invariant factor $2b$ is exactly $2$, i.e., $b=1$.
\end{definition}
Analogous to Proposition \ref{detsnf}, the set $\mathcal{F}_n$ (and $\mathcal{F}_n^{*}$) can also be described using determinant condition.
\begin{proposition}\label{detsnf2}
	Let $G\in \mathcal{H}_n^s$ and $\tau$ be either of the twins. Let $b$ be odd, square-free and positive. Then
	\begin{equation}\label{defb}
\frac{\det W_{\tau, n}(G)}{2^{\lfloor\frac{n}{2}\rfloor-1}}=\pm b
	\end{equation} if and only if the SNF of $W(G)$ is
	$$\diag[\underbrace{1,1,\ldots,1}_{\lceil\frac{n}{2}\rceil},\underbrace{2,2,\ldots,2,2b,0}_{\lfloor\frac{n}{2}\rfloor}].$$
\end{proposition}
\begin{proof}
	Let $W'=[e,A(G)e,\ldots,A^{n-2}(G)e,0]$, which is obtained from $W$ by replacing the last column by zero vector. As $\r W=n-1$,  Lemma~\ref{intcom} implies that $A^{n-1}(G)e$ can be written as a linear combination of $e,A(G)e,\ldots,A^{n-2}(G)e$ with integral coefficients. This means that $W$ and $W'$ are unimodular equivalent, that is, $W$ and $W'$ must have the same Smith normal form. Note that $W'$ has two equal rows corresponding to the twin vertices. We find that all $(n-1)\times (n-1)$ minor of $W'$ is either $0$ or $\pm\det W_{\tau n}(G)$ and hence $D_{n-1}(W')=|\det W_{\tau, n}(G)|$. Thus, if the SNF of $W(G)$ (or equivalently of $W'$) has the given structure, then $D_{n-1}(W')=2^{\lfloor\frac{n}{2}\rfloor-1}b$ and hence  $\frac{\det W_{\tau, n}(G)}{2^{\lfloor\frac{n}{2}\rfloor-1}}=\pm b$.
	
	Conversely, if $|\det W_{\tau n}(G)|=2^{\lfloor\frac{n}{2}\rfloor-1}b$ for some odd and square-free $b$, then  the SNF of $W'$ (or $ W(G)$) can be written as $S=\diag(1,\ldots,1,2^{l_1},2^{l_2},\ldots,2^{l_{t-1}},2^{l_t}b,0)$. By Lemma \ref{r2up}, we have $\br W(G)\le\lceil\frac{n}{2}\rceil$ and hence $t\ge \lfloor\frac{n}{2}\rfloor-1$. Moreover, we have $l_1+l_2+\cdots+l_{t}=2^{\lfloor\frac{n}{2}\rfloor-1}$ since $D_{n-1}(W')=2^{\lfloor\frac{n}{2}\rfloor-1}b$.
	It follows that $l_1=l_2=\cdots=l_t=1$ and $t=\lfloor\frac{n}{2}\rfloor-1$. This proves the proposition.
\end{proof}
It seems natural to guess that all graphs in $\mathcal{F}_n$ is DGS.
Unfortunately, this is \emph{not} true as shown by the following example.

\noindent
\begin{example}\label{countexp}
\textup{ Let $n=9$. Let $G$ be the graph whose adjacency matrix $A=A(G)$ is given as follows:
\begin{equation*}
A=\left[ \begin{array}{ccccccccc}
0&1&0&0&1&1&1&1&0\\
1&0&0&0&1&1&1&1&0\\
0&0&0&0&1&0&0&1&1\\
0&0&0&0&0&1&0&0&0\\
1&1&1&0&0&0&1&0&0\\
1&1&0&1&0&0&1&1&0\\
1&1&0&0&1&1&0&0&1\\
1&1&1&0&0&1&0&0&1\\
0&0&1&0&0&0&1&1&0
\end{array}\right].
\end{equation*}
It can be easily computed that the SNF of $W(G)$ is $\diag[1,1,1,1,1,2,2,2\times 3\times 101,0]$. Note that $G$ contains the first two vertices as twin vertices. We see that $G\in \mathcal{F}_n$. Let
\begin{equation*}
Q=\frac{1}{3}\left[ \begin{array}{ccccccccc}
2& -1& -1& 1& 1& 1& 0& 0& 0\\
0& 0& 0& 0& 0& 0& 3& 0& 0\\
-1& 2& -1& 1& 1& 1& 0& 0& 0\\
1& 1& 1& 2& -1& -1& 0& 0& 0\\
1& 1& 1& -1& 2& -1& 0& 0& 0\\
-1& -1& 2& 1& 1& 1& 0& 0& 0\\
1& 1& 1& -1& -1& 2& 0& 0& 0\\
0& 0& 0& 0& 0& 0& 0& 3& 0\\
0& 0& 0& 0& 0& 0& 0& 0& 3
\end{array}\right].
\end{equation*}
It is easy to verify that $Q$ is a rational orthogonal matrix with level $\ell=3$, and $Q^\T AQ$ is an adjacency matrix of another graph. It follows from Corollary \ref{dgsperm} that $G$ is not DGS.}
\end{example}

Although the straightforward extension of Theorem \ref{controDGS} is incorrect, we find that, under an additional assumption, graphs in $\mathcal{F}_n$ can be shown to be DGS.
\begin{definition}\label{deflambda1}
	Let $G\in \mathcal{H}_n^s$. We define
	\begin{equation}\label{deflambda1}
\lambda_1=	\lambda_1(G)= \begin{cases} -1& \text{if the twins in $G$ are adjacent,}  \\
	0&\text{otherwise.}\\
	\end{cases}
	\end{equation}
\end{definition}
\begin{theorem}\label{mainapp}
	Let $G\in \mathcal{F}_n$ and $2b$ be the $(n-1)$-th invariant factor of $W(G)$. Then $G$ is DGS provided that
	\begin{equation}\label{addcon}
	N(W^\T(G))\not\subset N(A(G)-\lambda_1(G) I) \text{~over~} \mathbb{F}_p,
	\end{equation}
	 for each odd prime factor $p$ of $b$. In particular, every graph in $\mathcal{F}^*_n$ is DGS.
\end{theorem}
We remark that in Example \ref{countexp}, Condition (\ref{addcon}) does not hold for $p=3$. The proof of this theorem will be postponed to  the next section. The overall idea comes from \cite{wang2013EJC} and \cite{wang2017JCTB}, but difficulty inevitably arises due to the increase of the dimension of the nullspace $N(W^\T(G))$ by one. Before going into the details, we first give some structural properties on the two-dimensional nullspace $N(W^\T(G))$, which will results in an equivalent description of (\ref{addcon}).

Let $G\in \mathcal{H}_n^s$ with twin vertices $\tau$ and $\tau'$, where $\tau<\tau'$. We define $\alpha(G)=(a_1,a_2,\ldots,a_n)^\T$, where
\begin{equation}\label{defalpha}
a_k= \begin{cases} -1& \text{if $k=\tau$,}  \\
1&\text{if $k=\tau'$,}\\
0&\text{otherwise.}\\
\end{cases}
\end{equation}
The following result can be easily verified from definitions.

\begin{proposition}\label{eigQ}
Let $G\in H_{n}^s$ with twins $\tau$ and $\tau'$, where $\tau<\tau'$. Then 	$W^\T(G)\alpha(G)=0$, $A(G)\alpha(G)=\lambda_1(G)\alpha(G)$ and $\xi(G)= W_{\tau',n}(G)\alpha(G)=-W_{\tau, n}(G)\alpha(G).$
\end{proposition}

\begin{definition}\label{defhW}
	Let $G\in \mathcal{H}_n^s$. Write
$\hat{W}(G)=[e,A(G)e,\ldots,A^{n-2}(G)e,\alpha(G)].$
\end{definition}
\begin{proposition}\label{Walpha}
	Let $G\in \mathcal{F}_n$ and $2b$ be the $(n-1)$-th invariant factor of $W(G)$. Then $\pr \hat{W}(G)=n-1$ for any odd prime factor $p$ of $b$.
\end{proposition}
\begin{proof}
	Let $\tau$ and $\tau'$ be the twin vertices of $G$. Then row $\tau$ and row $\tau'$ of $W(G)$ are equal. Thus, by the definition of $\alpha(G)$, we see that $\alpha^\T(G)W(G)=0$ and hence  $\alpha^\T(G)W(G)\equiv 0\pmod{p}.$
	Also, by the definition of $\alpha(G)$, we have $\alpha^\T(G)\alpha(G)=2\not\equiv 0\pmod{p}$. It follows that $\alpha(G)$ does not belongs to the column space of $W(G)$ over $\mathbb{F}_p$. This indicates that $\pr [W(G),\alpha(G)]=1+\pr W(G)$. By the definition of $\mathcal{F}_n$, we have $\pr W(G)=n-2$ and hence $\pr[W(G),\alpha(G)]=n-1$. Finally, as  $\pr W(G)=n-2$, Lemma \ref{first} implies that the last two columns of  $W(G)$ can be written as linear combinations of the first $n-2$ columns. Thus we can safely remove the $(n-1)$-th column and/or the $n$-th column in $[W(G),\alpha(G)]$ without changing its rank. In particular,  $\pr \hat{W}(G)=n-1$, as desired.
\end{proof}
By Proposition \ref{Walpha}, we know that the solution space of $\hat{W}^\T(G) x\equiv 0\pmod{p}$ is one dimensional. We define $\beta(G;p)$, or simply $\beta(G)$, to be a nontrivial (i.e., nonzero mod $p$) integral solution in $\{0,1,\ldots,p-1\}^n$ such that the last nonzero entry of $\beta(G)$ is 1. We note that $\beta(G)$ is unique and hence is well-defined. In the following propositions, we always assume that $p$ is any odd prime factor of $b$ as in Proposition \ref{Walpha}.
\begin{proposition}\label{abindep}
	Let $G\in \mathcal{F}_n$. Then $\alpha(G)$ and $\beta(G)$ constitute a system of fundamental solutions to $W^\T(G)x\equiv 0\pmod{p}$. Moreover, $\alpha(G)^\T \beta(G)\equiv 0\pmod{p}.$
\end{proposition}
\begin{proof}
	Clearly, by the definitions of $\alpha(G)$ and $\beta(G)$, both $\alpha(G)$ and $\beta(G)$ are nontrivial solutions to $W^\T(G)x\equiv 0\pmod{p}$. Also, $\alpha^\T(G)\beta(G)\equiv 0\pmod{p}$. Note that $\pr W(G)=n-2$ and hence the solution space of $W^\T(G)x\equiv 0\pmod{p}$ is two dimensional. We are done if  $\alpha(G)$ and $\beta(G)$ are linearly independent over $\mathbb{F}_p$. Suppose to the contrary that  $\alpha(G)$ and $\beta(G)$ are linearly dependent. Then we have $\beta(G)\equiv t\alpha(G) \pmod{p}$ for some integer $t\not\equiv 0\pmod{p}$ as both $\alpha(G)$ and $\beta(G)$ are nonzero vectors over $\mathbb{F}_p$. Consequently, we have $\alpha(G)^\T \beta(G)\equiv t\alpha(G)^\T\alpha(G)=2t\not\equiv 0\pmod{p}$. This contradiction completes the proof of the proposition.
\end{proof}
\begin{proposition}\label{eigbeta}
	Let $G\in \mathcal{F}_n$. Then there exists an integer $\lambda_0=\lambda_0(G;p)$ such that \begin{equation}\label{def0}
	A(G)\beta(G)\equiv\lambda_0\beta(G)\pmod{p}.
	\end{equation}
\end{proposition}
\begin{proof}
	As $G$ and $p$ are fixed in the following argument, we use simplified notations, for example $W$ means $W(G)$.	We first show that $N(W^\T)$ is $A$-invariant over $\mathbb{F}_p$. Let $x$ be any vector in $N(W^\T)$, that is, $W^\T x\equiv 0\pmod{p}$. Now we have $e^\T A^i x\equiv 0\pmod{p}$ for $i=0,1,\ldots,n-1$. By Cayley-Hamilton Theorem, we know that $A^n$ can be written as a linear combination of $A^0,A^1,\ldots,A^{n-1}$. Thus we have $e^\T A^n x\equiv 0\pmod{p}$ and hence  $W^\T Ax\equiv 0\pmod{p}$. This proves that $N(W^\T)$ is $A$-invariant.
	
	By Proposition \ref{abindep}, we have $N(W^\T)=\span(\alpha,\beta)$. As $N(W^\T)$ is $A$-invariant, there exist two integers $c_1$ and $c_2$ such that $A\beta\equiv c_1\alpha+c_2\beta\pmod{p}$. Then, as  $\alpha^\T\alpha=2$ and $\alpha^\T\beta\equiv 0\pmod{p}$, we have
	\begin{equation}
	\alpha^\T A\beta\equiv c_1\alpha^\T\alpha+c_2\alpha^\T\beta\equiv 2c_1\pmod{p}.
	\end{equation}
	On the other hand, as $A\alpha=\lambda_1\alpha$ and $A$ is symmetric, we also have
	\begin{equation}
	\alpha^\T A\beta=\beta^\T A\alpha\equiv\lambda_1\beta^\T \alpha\equiv 0\pmod{p}.
	\end{equation}
	It follows that  $2c_1\equiv 0\pmod{p}$, i.e., $c_1\equiv 0\pmod{p}$ as $p$ is odd. Therefore $A\beta\equiv c_2\beta\pmod{p}$. This proves the proposition.
\end{proof}
In the following, we shall use $\lambda_0$, or more precisely $\lambda_0(G;p)$, to denote an integer such that (\ref{def0}) holds. That is, $\lambda_0(G;p)$ is the eigenvalue of $A$ corresponding to $\beta(G;p)$ over $\mathbb{F}_p$. Of course $\lambda_0(G;p)$ is unique up to congruence modulo $p$ and we may safely assume $\lambda_0(G;p)\in\{0,1,\ldots,p-1\}$.

Now we can give an equivalent description of (\ref{addcon}).
\begin{proposition}\label{neqeig}
	Using the notations of Theorem \ref{mainapp}, we have (\ref{addcon}) holds if and only if $\lambda_1(G)\not\equiv \lambda_0(G;p)\pmod{p}$.
\end{proposition}
\begin{proof}
By Proposition \ref{abindep}, we have $N(W^\T (G))=\span(\alpha(G),\beta(G;p))$ over $\mathbb{F}_p$. By Proposition \ref{eigQ}, we have $A(G)\alpha(G)=\lambda_1(G)\alpha(G)$ and of course $A(G)\alpha(G)\equiv \lambda_1(G)\alpha(G)\pmod{p}$. That is, $\alpha(G)\in N(A(G)-\lambda_1(G)I)$ over $\mathbb{F}_p$. Therefore, we find that (\ref{addcon}) holds if and only if
\begin{equation}\label{addcon1}
\beta(G;p)\not\in N(A(G)-\lambda_1(G)I)\text{~over~} \mathbb{F}_p.
\end{equation}
Clearly, (\ref{addcon1}) is equivalent to $A(G)\beta(G;p)\not\equiv\lambda_1(G)\beta(G;p)\pmod{p}$.
On the other hand, by Proposition \ref{eigbeta}, we have $A(G)\beta(G;p)\equiv \lambda_0(G;p)\beta(G;p)\pmod{p}$. Thus,  (\ref{addcon1}) holds if and only if $\lambda_1(G)\not\equiv \lambda_0(G;p)\pmod{p}$. This completes the proof of this proposition.
\end{proof}
\section{Proof of Theorem~\ref{mainapp}}
In this section, we present the proof of Theorem~\ref{mainapp}. Before doing so, we need several results below.
\begin{proposition}\label{elldivdn}
	Let $G\in \mathcal{H}_n^s$ and $Q\in \mathcal{Q}_G$ with level $\ell$. Then
	$\ell\mid d_n$ and hence $\ell\mid\det W_0(G)$, where $d_n$ is the $n$-th invariant factor of $W_0(G)$.
\end{proposition}
\begin{proof}
	Let $H$ be a graph such that $Q\in \mathcal{Q}_G(H)$. By Theorem \ref{main}, $Q=Q_0$ or $Q=Q_1$, where $Q_0=W_0(G)W_0^{-1}(H)$ and $Q_1=W_1(G)W_0^{-1}(H)$. As $G\in \mathcal{H}_n^s$, the two matrices $Q_0$ and $Q_1$ must have the same level by Proposition \ref{samelevel}. Thus, we may assume that $Q=Q_0$. As $Q_0$ is orthogonal, we have $Q_0^\T=Q_0^{-1}$ and hence $Q_0^\T=W_0(H)W_0^{-1}(G)$. Clearly, $\ell(Q_0)=\ell(Q_0^\T)$ and hence we have $\ell=\ell(Q_0)=\ell(W_0(H)W_0^{-1}(G))$.
	
	Write $W_0(G)=U\diag(d_1,d_2,\ldots,d_n)V$ where $U,V$ are unimodular and $d_1,d_2,\ldots,d_n$ are invariant factors of $W_0(G)$. Note that
	$$d_nW_0(H)W_0^{-1}(G)=W_0(H)V^{-1}\diag\left(\frac{d_n}{d_1},\frac{d_n}{d_2},\ldots,\frac{d_n}{d_n}\right)U^{-1}.$$
	Thus $d_nW_0(H)W_0^{-1}(G)$ is integral as all four factors in the right hand  are integral matrices. Therefore, by the minimality of $\ell(W_0(H)W_0^{-1}(G))$, we must have $\ell(W_0(H)W_0^{-1}(G))\mid d_n$, that is, $\ell\mid d_n$.  This proves the proposition.
\end{proof}
\begin{proposition}\label{sameodd}
	Let $G\in \mathcal{F}_n$ with twin vertices $\tau$ and $\tau'$. Then
	$\det W_0=2^{\lfloor\frac{n}{2}\rfloor} b^2$ and
	$\det \hat{W}(G)=\pm 2^{\lfloor\frac{n}{2}\rfloor} b$, where $b$ is defined in (\ref{defb}).
\end{proposition}
%

\begin{proof}
 Expanding $\det W_0(G)$ by the last column and using the last assertion of Proposition \ref{eigQ}, we have
		\begin{equation*}\label{detw0}\det W_0(G)=2^{1-\lfloor\frac{n}{2}\rfloor}\xi^\T (G)\xi(G)=2^{1-\lfloor\frac{n}{2}\rfloor}W_{\tau,n}^2(G)\alpha^\T\alpha=2^{2-\lfloor\frac{n}{2}\rfloor}W_{\tau,n}^2(G).
		\end{equation*}
	By (\ref{defb}), $W_{\tau,n}(G)=\pm 2^{\lfloor\frac{n}{2}\rfloor -1}b$ and consequently, we have $\det W_0(G)=2^{\lfloor\frac{n}{2}\rfloor}b^2.$
	
	Finally, by the definition of $\hat{W}(G)$, one easily sees  that $\hat{W}(G)$ can be obtained from $W_0(G)$ after dividing the last column by $2^{1-\lfloor\frac{n}{2}\rfloor}W_{\tau,n}(G)$. Thus, $\det \hat{W}(G)=2W_{\tau,n}(G)=\pm 2^{\lfloor\frac{n}{2}\rfloor}b$. This completes the proof of this proposition.
\end{proof}

An immediate consequence of the two preceding propositions is the following corollary.
\begin{corollary}\label{pdn1}
		Let $G\in \mathcal{H}_n^s$ and $Q\in \mathcal{Q}_G$ with level $\ell$. Then each odd prime factor of  $\ell$ is a factor of $b$, where $b$ is defined in (\ref{defb}).
	\end{corollary}
\subsection{The case $p$ is odd}\label{oddp}
Let $G\in \mathcal{F}_n$ and $2b$ be the $(n-1)$-th invariant factor of $W(G)$. We assume that $p$ is an odd prime factor of $b$ in this subsection. The main aim of this subsection is to show the following theorem.
\begin{theorem}\label{mainodd}
		Let $G\in \mathcal{F}_n$ and $Q\in \mathcal{Q}_G$ with level $\ell$. If (\ref{addcon}) holds then $p\nmid \ell$.
\end{theorem}
We begin with the  following general result which illustrates basic properties of $Q\in\mathcal{Q}_G$ for $G\in \mathcal{F}_n$.  We remind the reader that this result does not rely on (\ref{addcon}).
\begin{proposition}\label{mainp}
	Let $G\in \mathcal{F}_n$ and $Q\in \mathcal{Q}_G$ with  level $\ell$.  	Let $\tau$ and $\tau'$ be twin vertices of $G$ and let $P$ be the permutation matrix obtained from $I_n$ by swapping row $\tau$ and row $\tau'$. If $p$ is an odd prime factor of $\ell$, then the followings hold.
	
\noindent	\textup{(\romannumeral1)}	$-\frac{p+1}{2}\beta^\T(G)\beta(G)$ is a quadratic residue modulo $p$. Let $ c_0$ and $-c_0$ be the two solutions of
\begin{equation}\label{qr}
t^2\equiv  -\frac{p+1}{2}\beta^\T(G)\beta(G) \pmod{p}
\end{equation}
and write $\gamma_\delta(G)=\beta(G)+(-1)^\delta c_0\alpha(G)$ for $\delta=0,1$.

\noindent	\textup{(\romannumeral2)}
$P\gamma_\delta(G)=\gamma_{1-\delta}(G)$ for $\delta=0,1$.

\noindent	\textup{(\romannumeral3)}
 The column space of $\ell Q$, over $\mathbb{F}_p$, is spanned by either $\gamma_0(G)$ or $\gamma_1(G)$. In particular,  $\pr \ell Q=1$.

\noindent	\textup{(\romannumeral4)}
$A(G)\gamma_\delta(G)\equiv  \lambda_0\gamma_\delta(G)\pmod{p}$, for $\delta=0,1$.

\noindent	\textup{(\romannumeral5)}
$\gamma^\T_\delta(G) (A(G)-\lambda_0I_n)\gamma_\delta(G)\equiv 0\pmod{p^2}$,  for $\delta=0,1$.

\end{proposition}
\begin{proof}
		Let $H$ be the graph such that $Q\in \mathcal{Q}_G(H)$. As $Q^\T A(G)Q=A(H)$, we have $Q^\T W(G)=W(H)$, i.e., $W^\T(G)Q =W^\T(H)$.
		
		Let $\hat{Q}=\ell Q$. By the definition of $\ell$, the matrix $\hat{Q}$ is integral and contains a column $q$ such that $q\not\equiv 0\pmod{p}$. As $\hat{Q}^\T \hat{Q}=\ell^2 I_n$ and $p\mid \ell$, we have $q^\T q=\ell^2\equiv 0\pmod{p^2}$ and hence $q^\T q\equiv 0\pmod{p}.$ As $W^\T(G)\hat{Q}=\ell W^\T(H)\equiv 0\pmod{p}$, we see that $W^\T(G)q\equiv 0\pmod{p}$. It follows from  Proposition \ref{abindep} that there exist two integers $c_1$ and $c_2$ such that $q\equiv c_1\alpha(G)+c_2\beta(G)\pmod{p}$. We claim that $c_2\not\equiv 0\pmod{p}$. If $ c_2\equiv 0\pmod{p}$, that is, $q\equiv c_1\alpha(G)\pmod{p}$, then $q^\T q\equiv c_1^2\alpha^\T(G) \alpha(G)\equiv 2c_1^2\pmod{p}$ and hence $c_1\equiv 0\pmod{p}$ as $q^\T q\equiv 0\pmod{p}$. Consequently, $q\equiv 0\pmod{p}$, contradicting  our choice of $q$. This proves the claim.
		
		As $c_2\not\equiv 0\pmod{p}$, there exists an integer $\bar{c}_2$ such that $c_2\bar{c}_2\equiv 1\pmod{p}$. Let $\bar{q}=\bar{c}_2 q$. Then we have $\bar{q}\equiv \bar{c}_2c_1\alpha(G)+\beta(G)\pmod{p}.$ As $\alpha^\T(G)\alpha=2$ and $\alpha^\T(G)\beta(G)\equiv 0\pmod{p}$, we have $\bar{q}^\T\bar{q}\equiv 2(\bar{c}_2c_1)^2+\beta^\T(G)\beta(G)\pmod{p}$. On the other hand, as $q^\T q\equiv 0\pmod{p}$, we have $\bar{q}^\T\bar{q}=(\bar{c}_2)^2q^\T q\equiv 0\pmod{p}$. Therefore,
		$2(\bar{c}_2c_1)^2+\beta^\T(G)\beta(G)\equiv 0\pmod{p}$, that is, $$(\bar{c}_2c_1)^2\equiv -\frac{p+1}{2}\beta^\T(G)\beta(G)\pmod{p}.$$
		This proves (\romannumeral1).

		Write $\beta(G)=(b_1,b_2,\ldots,b_n)^\T$. As $\alpha^\T(G)\beta(G)\equiv 0\pmod{p}$, we easily see from (\ref{defalpha}) that $b_\tau\equiv b_{\tau'}\pmod{p}$ and indeed  $b_\tau=b_{\tau'}$ since we have restricted each entry of $\beta(G)$ in $\{0,1,\ldots,p-1\}$. Now one  finds that the two vectors  $\beta(G)+c_0\alpha(G)$ and $\beta(G)-c_0\alpha(G)$ can be obtained from each other by swapping  the $\tau$-th entry and  $\tau'$-th entry. In other words, $\gamma_1(G)=P\gamma_0(G)$ and (\romannumeral2) is proved.
		
		Let $q_1,q_2,\ldots,q_s$ be all columns of $\hat{Q}$ which are nonzero vectors over $\mathbb{F}_p$. By the same argument as in the proof of (\romannumeral1), there exist integers $k_1,k_2,\ldots,k_s$, each nonzero modulo $p$, such that
		\begin{equation}\label{spm}
	k_i q_i\equiv \pm c_0\alpha(G)+\beta(G)\pmod{p}, i=1,2,\ldots,s,
		\end{equation} where $c_0$ is a solution to (\ref{qr}). Note that  (\romannumeral3) trivially holds if $s=1$ or $c_0\equiv 0\pmod{p}$. Thus, we may assume $s\ge 2$ and  $c_0\not\equiv 0\pmod{p}$. To show (\romannumeral3), it suffices to show that the signs must be consistent throughout all congruences in (\ref{spm}). Suppose to the contrary that there exist two indices, say 1 and 2, such that $k_1 q_1\equiv  +c_0\alpha(G)+\beta(G)\pmod{p}$ but  $k_2 q_2\equiv  -c_0\alpha(G)+\beta(G)\pmod{p}$. As $\hat{Q}^\T\hat{Q}=\ell^2I_n\equiv 0\pmod{p}$, we have $q_1^\T q_1\equiv q_2^\T q_2\equiv q_1^\T q_2\equiv 0\pmod{p}$  and hence $(k_1q_1-k_2q_2)^\T(k_1q_1-k_2q_2)\equiv 0\pmod{p}$. Noting that $k_1q_1-k_2q_2\equiv 2c_0\alpha(G) $ and $\alpha^\T(G)\alpha(G)=2$, we can obtain that $8c_0^2\equiv 0\pmod{p}$, that is, $c_0\equiv 0\pmod{p}$ as $p$ is an odd prime. This contradicts our assumption and hence completes the proof of (\romannumeral3).
		
		As $\pr \hat{Q}=1$ and $q$ is a nonzero column vector in $\hat{Q}$,  we can write $$\hat{Q}\equiv q(m_1,m_2,\ldots,m_n)\pmod{p},$$ for some integers $m_1,m_2,\ldots,m_n$. We use $k$ to denote the column index corresponding to $q$ in $\hat{Q}$. Let $h=(h_1,h_2,\ldots,h_n)^\T$ be the $k$-th column of $A(H)$. Noting $$A(G)\hat{Q}=\hat{Q}A(H)\equiv q(m_1,m_2,\ldots,m_n)A(H) \pmod{p},$$ and taking the $k$-th columns on both sides, we obtain that $A(G)q\equiv \lambda q \pmod{p}$, where $\lambda=m_1h_1+m_2h_2+\cdots+m_nh_n$.
		By (\romannumeral3), either $q\in\span (\gamma_0(G))$ or $q\in\span (\gamma_1(G))$. Without loss of generality, we assume $q\in\span (\gamma_0(G))$. Then we have $
	A(G)\gamma_0(G)\equiv \lambda \gamma_0(G)\pmod{p}.$  Consequently, as  $A(G) P=PA(G)$, we have
	\begin{equation}\label{ap}
	A(G)P\gamma_0(G)=PA(G)\gamma_0(G)\equiv \lambda P\gamma_0(G)\pmod{p}.
	\end{equation}
	By (\romannumeral2), $P\gamma_0(G)=\gamma_1(G)$ and hence  (\ref{ap}) can be written as  $A(G)\gamma_1(G)\equiv \lambda\gamma_1(G)\pmod{p}$. To complete the proof of  (\romannumeral4), it remains to show that $\lambda\equiv \lambda_0\pmod{p}$. First consider the case that $\lambda_1(G)\equiv \lambda_0\pmod{p}$. That is, $\alpha(G)$ and $\beta(G)$ correspond to the same eigenvalue $\lambda_0$. Since $\gamma_0(G)\in \span(\alpha(G),\beta(G))$, we must have $A(G)\gamma_0(G)\equiv \lambda_0 \gamma_0(G)\pmod{p}$. This indicates that $\lambda\equiv \lambda_0\pmod{p}$ for this case. Next consider the case $\lambda_1(G)\not\equiv \lambda_0\pmod{p}$. We claim that $\gamma_0(G)=\beta(G)$, that is, $c_0\equiv 0\pmod{p}$. Indeed, if $c_0\not\equiv 0\pmod{p}$ then $\beta(G)+c_0\alpha(G)$ can never be an eigenvector of $A(G)$ since $\alpha(G)$ and $\beta(G)$ correspond to two distinct eigenvalues $\lambda_1(G)$ and $\lambda_0$ over $\mathbb{F}_p$. This contradicts the fact that $A\gamma_0(G)\equiv \lambda\gamma_0(G)\pmod{p}$. Thus, $\gamma_0(G)=\beta(G)$ and hence $\lambda\equiv\lambda_0(G)\pmod{p}$. This completes the proof of (\romannumeral4).
	
	As $q\in \span(\gamma_0(G))$ and both $q$ and $\gamma_0(G)$ are nonzero vectors over $\mathbb{F}_p$, we can write $\gamma_0\equiv tq \pmod{p}$, i.e., $\gamma_0=tq+p\eta$ for some integer $t$ and integral vector $\eta$. Moreover, as $A(G)\gamma_0(G)\equiv \lambda_0 \gamma_0(G)\pmod{p}$ by (\romannumeral4), we have
	$A(G)q\equiv \lambda_0 q\pmod{p}$, that is,
	\begin{equation}\label{A0q}
	(A(G)-\lambda_0I_n)q\equiv 0 \pmod{p}.
				\end{equation}
	Noting that $\hat{Q}^\T \hat{Q},\hat{Q}^\T A(G)\hat{Q}\equiv 0\pmod{p^2}$ and $q$ is a column of $\hat{Q}$, we have $q^\T q,q^\T A(G) q\equiv 0\pmod{p^2}$ and hence
	\begin{equation}\label{qA0q}
	q^\T (A(G)-\lambda_0 I_n)q\equiv 0\pmod{p^2}.
	\end{equation}
	 Write $A_0=A(G)-\lambda_0 I_n$. Note that $A_0$ is symmetric.	It follows from (\ref{A0q}) and (\ref{qA0q}) that
	\begin{eqnarray*}
	\gamma^\T_0(G) A_0\gamma_0(G)&=&(tq+p\eta)^\T A_0(tq+p\eta)\\
	&=&t^2q^\T A_0q+2tp\cdot \eta^\T A_0 q+p^2\eta^\T A_0\eta^\T\\
	&\equiv &0\pmod{p^2}.
	\end{eqnarray*}
Finally, as $\gamma_1(G)=P\gamma_0(G)$ and $P^\T A_0P=A_0$,  we see that $\gamma^\T_1(G) A_0\gamma_1(G)=\gamma^\T_0(G) A_0\gamma_0(G).$
This proves (\romannumeral5).
\end{proof}
The conclusion of Proposition \ref{mainp} can be much simplified assuming (\ref{addcon}).
\begin{corollary}\label{mainpforsym}
	Under the assumption of Proposition \ref{mainp}, if (\ref{addcon}) holds, then we have
	
	\noindent	\textup{(\romannumeral1)}	 $\beta^\T(G;p)\beta(G;p)\equiv 0\pmod{p}$, and
	
	\noindent	\textup{(\romannumeral2)}	
	$\beta^\T(G;p)(A(G)-\lambda_0I_n)\beta(G;p)\equiv 0\pmod{p^2}$.
\end{corollary}
\begin{proof}
As (\ref{addcon}) holds, we have $\lambda_1(G)\not\equiv \lambda_0(G;p)\pmod{p}$ by Proposition \ref{neqeig}. From the proof of Proposition \ref{mainp}(\romannumeral4) we know that $\gamma_0(G)=\beta(G)$ or equivalently, $c_0\equiv 0\pmod{p}$. Since $c_0$ is a solution of (\ref{qr}), we find that $\beta^\T(G)\beta(G)\equiv 0\pmod{p}$.  Thus (\romannumeral1)  holds.  The remaining assertion  follows from Proposition \ref{mainp}(\romannumeral5) since $\gamma_0(G)=\beta(G;p)$.
\end{proof}
\begin{lemma}\label{partwy}
	If $(A(G)-\lambda_0 I)y\equiv p^j(l \alpha(G)+m\beta(G;p))\pmod{p^{j+1}}$ for some $l,m\in \mathbb{Z}$ and $j\in \{0\}\cup \mathbb{Z}^+$, then
	\begin{equation}
	\hat{W}^\T(G)y\equiv e^\T y(1,\lambda_0,\ldots,\lambda_0^{n-2},0)^\T+\alpha^\T(G)y(0,0,\ldots,0,1)^\T\pmod{p^{j+1}}.
	\end{equation}
\end{lemma}
\begin{proof}
	Note that the last row of $	\hat{W}^\T(G)y$ is $\alpha^\T y$. It suffices to show that $e^\T A^ky\equiv \lambda_0^ke^\T y\pmod{p^{j+1}}$ for $k=0,1,\ldots,n-2$. Clearly  $e^\T A^ky\equiv \lambda_0^ke^\T y\pmod{p^{j+1}}$ holds for $k=0$. Now assume that $e^\T A^ky\equiv \lambda_0^ke^\T y\pmod{p^{j+1}}$ holds for $k<n-2$ and we are going to check it for $k+1$.  By Proposition \ref{abindep}, we have $W^\T\alpha,W^\T\beta\equiv 0\pmod{p}$ and hence	$e^\T A^k\alpha,e^\T A^k\beta\equiv 0\pmod{p}$. Now, by the condition of this lemma and induction hypothesis, we have
	\begin{equation}
	e^\T A^{k+1}y\equiv e^\T A^k(\lambda_0 y+p^j l\alpha+p^j m\beta)\equiv \lambda_0^{k+1}e^\T y\pmod{p^{j+1}}.
	\end{equation}
	This proves the lemma.
\end{proof}
\begin{lemma}\label{prn1}
	Under the assumption of Proposition \ref{mainp}, if (\ref{addcon}) holds, then we have
	
	\noindent	\textup{(\romannumeral1)}
	$\pr[A(G)-\lambda_0 I,\alpha(G),\beta(G;p)]=n-1$, and
	
	\noindent	\textup{(\romannumeral2)}
	$\dim L= 3$, where $L=\{y\in \mathbb{F}_p^n\colon\,(A(G)-\lambda_0I)y\in \span(\alpha(G),\beta(G;p))\}$.
\end{lemma}
\begin{proof}
	For simplicity, we shall write $A=A(G)$, $\hat{W}=\hat{W}(G)$, $\alpha=\alpha(G)$ and $\beta=\beta(G;p)$. By Proposition \ref{abindep}, Proposition \ref{eigbeta} and Corollary  \ref{mainpforsym}(\romannumeral1), we have $\beta^\T[A-\lambda_0I,\alpha,\beta]\equiv 0\pmod{p}$. It follows from the fact $\beta\not\equiv 0\pmod{p}$ that $\pr [A-\lambda_0 I,\alpha,\beta]\le n-1$. Thus, to show (\romannumeral1) it suffices to prove the inverse inequality.
	
As $A\alpha=\lambda_1 \alpha$ we have  $(A-\lambda_0 I)\alpha=(\lambda_1-\lambda_0)\alpha\in \span(\alpha,\beta)$ and hence  $\alpha\in L$ by the definition of $L$. Also, $\beta\in N(A-\lambda_0 I)\subset L$ by Proposition \ref{eigbeta}. Thus, $L\supset \span(\alpha,\beta)$ and hence $\dim L\ge 2$.
	
	Let $K=\{y\in L\colon\,e^\T y=0,\alpha^\T y=0~ \textup{over}~ \mathbb{F}_p\}$. Then $\dim K\ge\dim L-2$. Using Lemma \ref{partwy} for the special case $j=0$, we find that $K\subset N(\hat{W}^\T).$ By Proposition \ref{Walpha}, $\pr \hat{W}=n-1$ and hence $\dim N(\hat{W}^\T)=1$. Thus, $\dim K\le 1$ and we must have $\dim L\le 3$. In particular, since  $N(A-\lambda_0 I)\subset L$ by the definition of $L$, it necessarily holds that $\dim N(A-\lambda_0 I)\le 3$, i.e., $\pr (A-\lambda_0 I)\ge n-3$.
	
	 We claim that  $\pr (A-\lambda_0 I)\neq n-3$. Suppose to the contrary that  $\pr (A-\lambda_0 I)=n-3$. Then, 	as $\dim N(A-\lambda_0 I)=3$, $N(A-\lambda_0 I)\subset L$ and $\dim L\le 3$, we must have $L= N(A-\lambda_0 I)$ and $\dim L=3$.
	 Noting that $\alpha\in L$, we see that $\alpha\in N(A-\lambda_0 I)$. Thus, $N(A-\lambda_0 I)\supset \span(\alpha,\beta)$, that is, $N(A-\lambda_0 I)\supset N(W^\T) $. This contradicts (\ref{addcon}) and hence proves the claim.  Thus, it suffices to consider the following two cases.
	
	\noindent\textbf{Case 1}:  $ \pr (A-\lambda_0 I)= n-1$.
	
	Clearly $\pr[A-\lambda_0 I,\alpha,\beta]\ge n-1$, so (\romannumeral1) holds. Thus,  $ \pr [A-\lambda_0 I,\beta]= \pr (A-\lambda_0 I)=n-1$ and hence  $\beta=(A-\lambda_0 I)y$ over $\mathbb{F}_p$ for some vector $y$. We claim that $y\not\in \span(\alpha,\beta)$ (over $\mathbb{F}_p$). Suppose to the contrary that $y\equiv k_1 \alpha+k_2\beta\pmod{p}$. Then, left-multiplying $A-\lambda_0 I$ to both sides, we obtain $\beta\equiv k_1(\lambda_1-\lambda_0)\alpha\pmod{p}$. This is a contradiction as $\beta$ and $\alpha$ are linearly independent over $\mathbb{F}_p$. Therefore, $y\not\in \span(\alpha,\beta)$ and hence $\pr[\alpha,\beta,y]=3$. Note that $y\in L$ by the definition of $L$. Thus, $\dim L\ge \pr[\alpha,\beta,y]=3$. This proves (\romannumeral2) as we always have  $\dim L\le 3$.

	\noindent\textbf{Case 2}: $ \pr (A-\lambda_0 I)= n-2$.
	
	As  $ \pr (A-\lambda_0 I)= n-2$ and $\beta\in N(A-\lambda_0 I)$, we can find another vector $\beta'$ such that $\beta,\beta'$ constitute a basis of $N(A-\lambda_0 I)$. By (\ref{addcon}),  we see that $\alpha\not\in N(A-\lambda_0 I)$, that is, $\alpha\not\in\span(\beta,\beta')$. Since   $\alpha,\beta,\beta'\in L$ and $\dim L\le 3$, we must have  $\dim L=3$ and $L=\span(\alpha,\beta,\beta')$. This proves (\romannumeral2).
	
	We claim that $\beta$ does not belong to the column space of $A-\lambda_0 I$. Suppose to the contrary that $\beta=(A-\lambda_0 I)y$ over $\mathbb{F}_p$ for some vector $y$. Then $y\in L$ but $y\not\in N(A-\lambda_0 I)$. As $\dim L=3$, we also have  $L=\span (\beta,\beta',y)$. Since $\alpha,\beta,\beta'$ is a basis of $L$, we can write $y\equiv k_1\alpha+k_2\beta+k_3\beta'\pmod{p}$ for some integers $k_1,k_2,k_3$. Left-multiplying $A-\lambda_0 I$ to both sides, we have
	$\beta\equiv k_1(\lambda_1-\lambda_0)\alpha\pmod{p}$, contradicting the fact that $\alpha$ and $\beta$ are linearly independent. This shows the claim and hence  $ \pr [A-\lambda_0 I,\beta]=n-1$. Thus,  $ \pr [A-\lambda_0 I,\alpha,\beta]\ge n-1$ and (\romannumeral1) follows.

	Combining Case 1 and Case 2, the lemma follows.
\end{proof}
\begin{lemma}\label{wxp2}
	Under the assumption of Proposition \ref{mainp}, if (\ref{addcon}) holds, then there exists a vector $x\not\equiv 0\pmod{p}$ such that $\hat{W}^\T(G) x\equiv 0\pmod{p^2}$.
\end{lemma}
\begin{proof}
	We also use the simplified notation as in the proof of  Lemma \ref{prn1}.
	By Corollary \ref{mainpforsym} (ii), $\beta^\T(A-\lambda_0I)\beta\equiv 0\pmod{p^2}$. That is,
	$\beta^\T\frac{(A-\lambda_0I)\beta}{p}\equiv 0\pmod{p}$. Note that $\beta^\T[A-\lambda_0I,\alpha,\beta]\equiv 0\pmod{p}$. By Lemma \ref{prn1}(\romannumeral1), $\pr [A-\lambda_0I,\alpha,\beta]=n-1$ and consequently each solution to $\beta^\T x\equiv 0\pmod{p}$ can be written as a linear combination of columns in $[A-\lambda_0I,\alpha,\beta]$. In particular, we can write
	\begin{equation}\label{mp}
	\frac{(A-\lambda_0I)\beta}{p}\equiv (A-\lambda_0 I)y+l\alpha+m\beta\pmod{p},
	\end{equation}
	for some vector $y$ and scalars $l,m$. Multiplying both sides by $p$ and moving  terms, we can rewrite (\ref{mp}) as
	\begin{equation}\label{mp2}
	(A-\lambda_0I)(\beta-py)\equiv p(l\alpha+m\beta)\pmod{p^2}.
	\end{equation}
	It follows from Lemma \ref{partwy} that
	\begin{equation}\label{mp3}
	\hat{W}^\T(\beta-py)\equiv  e^\T (\beta-py)(1,\lambda_0,\ldots,\lambda_0^{n-2},0)^\T+\alpha^\T (\beta-py)(0,0,\ldots,0,1)^\T\pmod{p^2}.
	\end{equation}
	Let $L=\{z\in \mathbb{F}_p^n\colon\,(A-\lambda_0I)z\in \span(\alpha,\beta)\}$ and $S$ be the two-dimensional subspace of $\mathbb{F}_p^n$ spanned by $(1,\lambda_0,\ldots,\lambda_0^{n-2},0)^\T$ and $(0,0,\ldots,0,1)$. By  Lemma \ref{partwy}, for each $z\in L$, we have $\hat{W} z\in S$. Consider a linear map $\sigma\colon\, L\mapsto S$ defined by $\sigma(z)=\hat{W}z$ for $v\in L$. By Proposition \ref{Walpha}, $\pr \hat{W}=n-1$ and hence  $\ker \sigma$ is at most one dimensional. In fact,  $\dim \ker \sigma=1$ since $\ker \sigma$ contains $\beta$, which is nonzero. By Lemma \ref{prn1}, $\dim L=3$. Therefore, $\dim\sigma(L)=\dim L-\dim \ker \sigma= 2$  and hence  $\sigma(L)=S$.
	
	Note that $e^\T \beta,\alpha^\T \beta\equiv 0\pmod{p}$. We  have $e^\T(\beta-py),\alpha^\T(\beta-py)\equiv 0\pmod{p}$. Write $c_1=\frac{e^\T(\beta-py)}{p}$ and $c_2=\frac{\alpha^\T(\beta-py)}{p}$. Then $c_1,c_2$ are integers. Letting $$\mu=c_1(1,\lambda_0,\ldots,\lambda_0^{n-2},0)^\T+c_2(0,0,\ldots,0,1)^\T,$$ we can simplify (\ref{mp3}) as
	\begin{equation}\label{mp4}
	\hat{W}^\T(\beta-py)\equiv   p \mu\pmod{p^2}.
	\end{equation}
	As $\sigma(L)=S$ and $\mu\in S$, we can write $\mu\equiv \hat{W}^\T \eta\pmod{p}$ for some vector $\eta$. Multiplying $p$ to both sides and combining with (\ref{mp4}) leads to
	$
	\hat{W}^\T(\beta-py)\equiv  \hat{W}^\T p\eta\pmod{p^2},
	$
	that is,
	\begin{equation}\label{mp5}
	\hat{W}^\T(\beta-py -p\eta)\equiv 0\pmod{p^2}.
	\end{equation}
	Let $x=\beta-py-p\eta$. Then $x\equiv \beta\pmod{p}$ and hence $x\not\equiv 0\pmod{p}$. This proves the lemma.	
\end{proof}

\begin{proof}[Proof of Theorem \ref{mainodd}]
 Suppose to the contrary that $p\mid\ell$. By Lemma \ref{wxp2}, $\hat{W}^\T x\equiv 0\pmod{p^2}$ has a solution $x\not\equiv 0\pmod{p}$. It follows from Fact \ref{basicsnf} (\romannumeral4) and (\romannumeral1) that $p^2\mid \det \hat{W}$. By Proposition \ref{sameodd}, we have $\det \hat{W}=\pm2^{\lfloor\frac{n}{2}\rfloor} b$ and consequently $p^2\mid b$ as $p$ is odd. This means that $b$ is not square-free, a contradiction. This completes the proof of Theorem \ref{mainodd}.
\end{proof}
\subsection{The case $p=2$}
The main idea comes from \cite{wang2017JCTB}. Set $k=\lfloor\frac{n}{2}\rfloor$ and $A=A(G)$. Following \cite{wang2017JCTB}, let $\tilde{W}$ and $\tilde{W}_1$ be the matrix defined as follows:
\begin{equation}
\tilde{W}= \begin{cases}
(e,Ae,\ldots,A^{k-1}e) &\mbox{if $n$ is even,}\\
(Ae,A^2e,\ldots,A^{k}e) &\mbox{if $n$ is odd,}
\end{cases}
\end{equation}
and
\begin{equation}
\tilde{W}_1= \begin{cases}
(e,A^2e,\ldots,A^{2k-2}e) &\mbox{if $n$ is even,}\\
(A^2e,A^4e\ldots,A^{2k}e) &\mbox{if $n$ is odd.}
\end{cases}
\end{equation}
\begin{proposition}\cite[Lemma 4.1]{wang2017JCTB}\label{mainfrom2}
	Let $G\in\mathcal{F}_n$  and $Q\in\mathcal{Q}_G$ with level $\ell$. If
	\begin{equation}\label{subset}
	\left\{v\colon\,\frac{W^\T(G)\tilde{W}_1}{2}v\equiv 0\pmod{2}\right\}\subset \left\{v\colon\,\tilde{W}v\equiv 0\pmod{2}\right\},
	\end{equation}
	then $\ell$ is odd.
\end{proposition}
\begin{remark}\textup{
	 Lemma 4.1 in \cite{wang2017JCTB} was restricted to controllable graphs. However the assumption of controllability is essentially only used for guaranteeing the rationality of $Q$ as in Theorem \ref{unique}.  Another major restriction in  \cite{wang2017JCTB} is that $\br W(G)=\lceil\frac{n}{2}\rceil$, which is satisfied for graphs in $\mathcal{F}_n$ introduced here. Note that we have already shown that each matrix in $\mathcal{Q}_G$ is rational for $G\in \mathcal{F}_n\subset \mathcal{H}_n$  in  Theorem \ref{main}.  The original argument for Lemma 4.1 is also valid in this new setting.}
\end{remark}

The following result is new and crucial, which can be regarded as an extension of Lemma 3.10 in \cite{wang2017JCTB} to noncontrollable graphs.
\begin{proposition}\label{fullbr}
	Let $G\in \mathcal{F}_n$. Then $\br \frac{W^\T(G)\tilde{W}_1}{2}=\lfloor\frac{n}{2}\rfloor$.
\end{proposition}
\begin{proof}
%
%

 Let $V$ be the matrix obtained from $W(G)$ by removing the last column, that is, $V=[e,Ae,\ldots,A^{n-2}e]$.
	
	\noindent\emph{Claim} 1: $2^{2\lfloor\frac{n}{2}\rfloor-1}\mid \det V^\T V$ but $2^{2\lfloor\frac{n}{2}\rfloor}\nmid\det V^\T V$.
	
	By Proposition \ref{vectxi}(\romannumeral3) and Proposition \ref{eigQ}, we have
	$$ \det V^\T V=\xi^\T (G)\xi(G)=2W_{\tau,n}^2(G).$$  By Proposition \ref{detsnf}, we know that $ W_{\tau,n}(G)=\pm2^{\lfloor\frac{n}{2}\rfloor-1}b$. Consequently,  $\det V^\T V=2^{2\lfloor\frac{n}{2}\rfloor-1}b^2$. This proves the claim as $b$ is odd.

	Now assume that $n$ is even. Then Lemma \ref{even}, together with the fact that $e^\T e$ is even, implies that $\frac{V^\T V}{2}$ is an integral matrix. By Claim 1, the determinant of  $\frac{V^\T V}{2}$ is odd and hence $\frac{V^\T V}{2}$ has full column rank over $\mathbb{F}_2$. This indicates that  $\frac{V^\T\tilde{W}_1}{2}$ also has full column rank as $\tilde{W}_1$ consists of some columns of $V$. Moreover, as $\frac{V^\T\tilde{W}_1}{2}$ can be  obtained from  $\frac{W^\T(G)\tilde{W}_1}{2}$ by removing the last row, we see that $\frac{W^\T(G)\tilde{W}_1}{2}$ necessarily has full column rank. This proves the proposition for even $n$.

	In the following we assume that $n$ is odd. Let $U=[e,Ae,\ldots,A^{n-3}e]$.
	
	\noindent\emph{Claim} 2: $\br \frac{U^\T A^2 U}{2}=n-2.$
	
	As $V=[e,AU]$, we have
	\begin{equation}
	V^\T V=\left[ \begin{array}{cc}
	n& e^\T AU \\
	U^\T Ae&U^\T A^2 U
	\end{array}\right].
	\end{equation}
	Define
	\begin{equation}
	M=\left[ \begin{array}{cc}
	n& e^\T AU \\
	\frac{1}{2}U^\T Ae&\frac{1}{2}U^\T A^2 U
	\end{array}\right],
	\end{equation}
	which is clearly an integral matrix by Lemma \ref{even}. 	As $n$ is odd, we see from Claim 1 that $2^{n-2}\mid \det V^\T V$ but $2^{n-1}\nmid \det V^\T V$. This indicates that
	$\det M$ is odd, that is $ M$ has full rank over $\mathbb{F}_2$. Note that the first row of $M$ is congruent to $(1,0,0,\ldots,0)$ over $\mathbb{F}_2$. We see that  $\frac{U^\T A^2 U}{2}$ also has full rank over $\mathbb{F}_2$. This proves Claim 2.
	
	Note that $A^2 U=[A^2 e,A^3 e,\ldots,A^{n-1}e]$ and $\tilde{W}_1=[A^2e,A^4e,\ldots,A^{n-1}e]$. One easily sees that  $\tilde{W}_1$ consists of some columns of $A^2 U$. Therefore, $\frac{U^\T \tilde{W}_1}{2} $ must have full column rank over $\mathbb{F}_2$ by Claim 2. Similarly,  as $\frac{U^\T\tilde{W}_1}{2}$ can be  obtained from  $\frac{W^\T(G)\tilde{W}_1}{2}$ by removing the last two rows, we see that $\frac{W^\T(G)\tilde{W}_1}{2}$ necessarily has full column rank. This proves the odd case and hence Proposition \ref{fullbr} follows.
\end{proof}

Now, we are ready to present the proof of Theorem~\ref{mainapp}.

\begin{proof}[{Proof of Theorem \ref{mainapp}}]
 Let $Q\in\mathcal{Q}_G$ and $\ell$ be its level. By Proposition \ref{fullbr}, $\frac{W^\T(G)\tilde{W}_1}{2}$ has full column rank over $\mathbb{F}_2$. Thus, (\ref{subset}) trivially follows as the left space contains only the zero vector. It follows from Proposition \ref{mainfrom2} that $\ell$ is odd.  Suppose to the contrary that $\ell$ has an odd prime factor, say $p$. Then by Corollary \ref{pdn1}, we have $p\mid b$ and consequently (\ref{addcon}) holds by the condition of this theorem. It follows from Theorem \ref{mainodd} that $p\nmid \ell$. This is a contradiction. Thus, we must have $\ell=1$. This indicates that $\mathcal{Q}_G$ contains only permutation matrices and hence $G$ is DGS by  Corollary \ref{dgsperm}. This finally completes the proof of Theorem \ref{mainapp}.
 \end{proof}

\section{Some examples}
In this section, we give some examples for illustrations.

\begin{example} \textup{Let $T_n$ and $U_n$ ($n\ge 2$) be two series of graphs as shown in Fig.~1. Using Mathematica, for small $n$, say $n<200$, we find that the determinants of $W_{1n}(T_n)$ and $W_{1n}(U_n)$ behave extremely regular, namely
\begin{equation}\label{conequ1}
\det W_{1n}(T_n)=
\begin{cases}
\pm 2^{\lfloor\frac{n}{2}\rfloor-1} &\text{if~$4\nmid n$},\\
0 & \text{if~$4\mid n$,}
\end{cases}
\end{equation}
and
\begin{equation} \label{conequ2}
\det W_{1n}(U_n)= \begin{cases}
\pm 2^{\lfloor\frac{n}{2}\rfloor-1} &\text{if~$3\nmid n$},\\
0& \text{if~$3\mid n$.}
\end{cases}
\end{equation}
By Proposition \ref{detsnf2}, we see that $T_n\in \mathcal{F}_n^{*}$ when $4\nmid n$, and $U_n\in \mathcal{F}_n^{*}$ when $3\nmid n$.  By Theorem \ref{mainapp}, such graphs are DGS.
}
\end{example}
 \begin{figure}[htbp]
	\begin{center}
		\includegraphics[height=2.6cm]{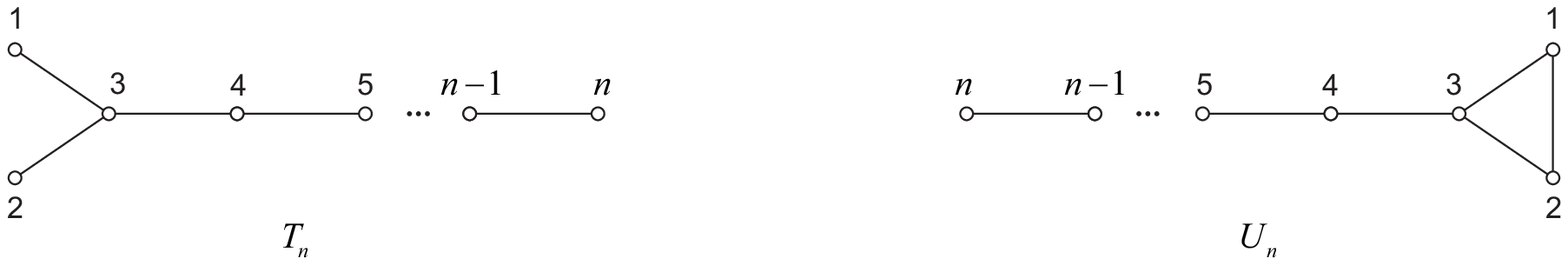}
		
		{\bf Fig.~1~}{Graphs $T_n$ and $U_n$}
	\end{center}
	\label{nK4}
\end{figure}
It is worth mentioning that for controllable graphs, Mao et al.~\cite{mao2015LAA} found a similar phenomenon in searching $n$-order trees ($n$ even) with $\det W(G)=2^{n/2}$; see Conjecture 1 in \cite{mao2015LAA}.  Although Eqs.~(\ref{conequ1}) and (\ref{conequ2}) seem simpler than the conjecture of Mao et al.,  direct proofs of the two equalities are still out of reach at this moment.

\begin{example}\textup{Let $n=10$. Let $G$ be a graph whose adjacency matrix $A=A(G)$ is given as follows:
		 \begin{equation*}
		 A=\left[ \begin{array}{cccccccccc}
		 0&1&0&1&0&0&0&0&0&0\\
		 1&0&0&1&0&0&0&0&0&0\\
		 0&0&0&0&1&0&1&1&1&1\\
		 1&1&0&0&1&1&0&1&1&0\\
		 0&0&1&1&0&0&0&1&1&0\\
		 0&0&0&1&0&0&1&0&0&0\\
		 0&0&1&0&0&1&0&1&0&1\\
		 0&0&1&1&1&0&1&0&0&1\\
		 0&0&1&1&1&0&0&0&0&1\\
		 0&0&1&0&0&0&1&1&1&0
		 \end{array}
		 \right].
		 \end{equation*}
It can be easily verified using Mathematica that the SNF of $W(G)$ is
		 $$\diag (1, 1, 1, 1, 1,2, 2, 2, 304690, 0).$$
Noting that the first two vertices are twins in $G$ and $304690=2\times 5\times 30469$, we see $G\in \mathcal{F}_n$.
As the twins are adjacent, we have $\lambda_1=-1$. Write $p_1=5$ and $p_2=30469$. Direct calculation shows that $\lambda_0(G;p_1)\equiv 2\pmod{p_1}$ and $\lambda_0(G;p_2)\equiv 1224\pmod{p_2}$. Thus $\lambda_0(G;p_i)\not\equiv \lambda_1\pmod{p_i}$ for $i=1,2$, that is, (\ref{addcon}) holds for both $p_1$ and $p_2$. Thus, $G$ is DGS by Theorem \ref{mainapp}.		
	}
\end{example}
Now we give an example to illustrate the possibility that even when (\ref{addcon}) fails, using Proposition \ref{mainp}, we may still be able to guarantee a graph $G\in\mathcal{F}_n$ to be DGS.
\begin{example} \textup{Let $G$ be a graph with adjacency matrix
\begin{equation*}
A=\left[ \begin{array}{ccccccccccccc}
0&1&0&0&0&1&1&1&1&1&0&0&1\\
1&0&0&0&0&1&1&1&1&1&0&0&1\\
0&0&0&1&0&0&0&0&1&0&0&1&1\\
0&0&1&0&0&0&1&1&1&1&1&0&0\\
0&0&0&0&0&0&0&1&1&1&1&1&0\\
1&1&0&0&0&0&0&0&1&0&0&1&1\\
1&1&0&1&0&0&0&1&0&0&1&1&1\\
1&1&0&1&1&0&1&0&0&0&0&0&0\\
1&1&1&1&1&1&0&0&0&0&0&1&0\\
1&1&0&1&1&0&0&0&0&0&1&1&0\\
0&0&0&1&1&0&1&0&0&1&0&0&0\\
0&0&1&0&1&1&1&0&1&1&0&0&0\\
1&1&1&0&0&1&1&0&0&0&0&0&0
\end{array}\right].
\end{equation*}
It can be easily verified using Mathematica that the SNF of $W(G)$ is
 $$\diag (1, 1, 1, 1, 1, 1, 1, 2, 2, 2, 2, 247799709690, 0).$$
Noting that $G$ contains the first two vertices as twins and  from the standard factorization $247799709690=2\times 3\times 5\times 13\times 3607\times 176153$, we see that $G\in \mathcal{F}_n$. Let $p_1,\ldots,p_5$ denote the five odd prime factors in increasing order. It can be verified that (\ref{addcon}) holds unless $p=p_2=5$.  Thus, we can not  exclude the possibility that $p_2\mid \ell$ by Theorem \ref{mainodd}, we need more calculations suggested by Proposition \ref{mainp}.}

 \textup{For the factor $p=p_2=5$, we obtain
 $\beta(G;p_2)=(2,2,1,0,0,2,2,4,0,2,1,3,1)$ and hence $-\frac{p_2+1}{2}\beta^\T(G;p_2)\beta(G;p_2)=-144\equiv 1\pmod{5}$. Now, we can take $c_0=1$ and
 $\gamma_0=(1,3,1,0,0,2,2,4,0,2,1,3,1)$.
 Note that $\lambda_1=-1$ and $\lambda_0(G;p)\equiv \lambda_1\pmod{5}$. We may take $\lambda_0=-1$ and then direct calculation shows that $\gamma_0^\T (A-\lambda_0 I_n)\gamma_0\not\equiv 0\pmod{5^2}$. Therefore, by Proposition \ref{mainp}, we must have $5\nmid \ell(Q)$ for any $Q\in \mathcal{Q}_G$. Thus, $G$ is DGS.}
\end{example}
 \section*{Acknowledgments}
This work is supported by the National Natural Science Foundation of China (Grant Nos.
\,12001006, 11401044, 11971376  and 11971406)
and the Scientific Research Foundation of Anhui Polytechnic University (Grant No.\,2019YQQ024).

\end{document}